\documentclass[reqno,11pt]{amsart}
\usepackage{amsmath,amssymb,latexsym,soul,cite,mathrsfs,amsfonts}
\usepackage{color,enumitem,graphicx}
\usepackage[colorlinks=true,urlcolor=blue,
citecolor=red,linkcolor=blue,linktocpage,pdfpagelabels,
bookmarksnumbered,bookmarksopen]{hyperref}
\usepackage[english]{babel}
\usepackage{mwe}
\usepackage{comment}
\usepackage{tikz}
\usepackage[left=2.9cm,right=2.9cm,top=2.8cm,bottom=2.8cm]{geometry}
\usepackage[hyperpageref]{backref}

\usepackage[colorinlistoftodos]{todonotes}
\makeatletter
\providecommand\@dotsep{5}
\def\listtodoname{List of Todos}
\def\listoftodos{\@starttoc{tdo}\listtodoname}
\makeatother

\usepackage{verbatim} 
\usepackage{tikz}

\numberwithin{equation}{section}

\def\cal{\mathcal}
\newtheorem{theorem}{Theorem}[section]
\newtheorem{proposition}[theorem]{Proposition}
\newtheorem{lemma}[theorem]{Lemma}
\newtheorem{corollary}[theorem]{Corollary}
\newtheorem{remark}[theorem]{Remark}

\newtheorem{example}{Example}

\newcommand\restr[2]{{
  \left.\kern-\nulldelimiterspace 
  #1 
  \vphantom{\big|} 
  \right|_{#2} 
  }}

\title[Some applications of the Nehari manifold method to functionals in $C^1(X \setminus \{0\})$]
{Some applications of the Nehari manifold method to functionals in $C^1(X \setminus \{0\})$}

\author[E.J.F. Leite]{Edir Júnior Ferreira Leite}
\author[H. Ramos Quoirin]{Humberto Ramos Quoirin}

\author[K. Silva]{Kaye Silva}

\address[E.J.F. Leite]{\newline\indent
	Departamento de Matem\'atica.   
	\newline\indent 
	Universidade Federal de S\~ao Carlos,
	\newline\indent
13565-905, S\~ao Carlos, SP, Brazil}
\email{\href{mailto:edirleite@ufscar.br}{edirleite@ufscar.br}}

\address[H. Ramos Quoirin]{\newline \indent CIEM-FaMAF \newline \indent Universidad Nacional de C\'{o}rdoba, \newline\indent 
(5000) C\'{o}rdoba, Argentina}
\email{\href{mailto:humbertorq@gmail.com}{ humbertorq@gmail.com}}

\address[K. Silva]{\newline\indent
	Instituto de Matem\'atica e Estat\'istica.   
	\newline\indent 
	Universidade Federal de Goi\'as,
	\newline\indent
Rua Samambaia, 74001-970, Goi\^ania, GO, Brazil}
\email{\href{mailto:kayesilva@ufg.br}{kayesilva@ufg.br}}


\begin{document}

\begin{abstract}
Given a real Banach space $X$, we show that the Nehari manifold method can be applied to functionals which are $C^1$ in $X \setminus \{0\}$. In particular we deal with functionals that can be unbounded near $0$, and prove the existence of a ground state and infinitely many critical points for such functionals. These results are then applied to three classes of problems: the {\it prescribed energy problem} for a family of functionals depending on a parameter, problems involving the {\it affine} $p$-Laplacian operator, and degenerate Kirchhoff type problems. \end{abstract}

\bigskip

\maketitle

\bigskip

\section{Introduction}
\medskip
Let $X$ be a uniformly convex (real) Banach space and $\Phi$ a  $C^1$ functional (in the sense of Fréchet) over $X$. It is well-known that the Nehari manifold method is one of the main tools in nonlinear analysis when looking for critical points of $\Phi$. In particular this method allows us to obtain ground states (or least energy critical points) as well as infinitely many  critical points under appropriate conditions on $\Phi$. We refer to  \cite{SW} for a description of these results (among many others) and its applications to boundary value problems. We aim at extending these results to functionals that are $C^1$ in $X \setminus \{0\}$. As a matter of fact, searching for nonzero critical points of $\Phi$ makes sense even when this functional is not defined at $0$, and dealing with the Nehari set associated to $\Phi$ seems to be a good choice in view of its definition:
 $$\mathcal{N}=\mathcal{N}(\Phi):=\{u \in X \setminus \{0\}: \Phi'(u)u=0\}$$

 Our motivation for dealing with $\Phi \in C^1(X \setminus \{0\})$ is not merely abstract, as this kind of functional arises in (at least) three classes of problems, namely: the {\it prescribed energy problem} for a family of functionals depending on a parameter, problems involving the {\it affine} $p$-Laplacian operator and degenerate Kirchhoff type problems. Let us describe them in the sequel.

\subsection{The prescribed energy problem}\label{sspep}
Let $I_1,I_2\in C^1(X)$ and  $\Phi_\lambda:=I_1-\lambda I_2$ with $\lambda$ a real parameter. We consider the problem
\begin{equation}\label{ef}
\Phi_\lambda'(u)=0, \quad \Phi_{\lambda}(u)=c,
\end{equation}
where $c$ is an arbitrary real number.
 By a (nontrivial) solution of \eqref{ef} we mean a couple $(\lambda,u) \in \mathbb{R} \times X \setminus \{0\}$ satisfying \eqref{ef}, in which case $u$ is a critical point of $\Phi_\lambda$ at the level $c$.  In this sense, a solution of \eqref{ef} yields a critical point of  $\Phi_\lambda$ with prescribed energy.  Our main goal here is to analyse the solvability of \eqref{ef} in terms of $I_1,I_2$ and $c \in \mathbb{R}$. We are mainly motivated by the study of  boundary value problems whose solutions are critical points of a given functional (the so-called {\it energy functional}).

The problem \eqref{ef} has been recently studied in \cite{I3,RSiS,RSS} by means of the {\it nonlinear generalized Rayleigh quotient } method (the NGRQ method, for short), which is based on the introduction of the functional $\lambda_c$ given through the following relation:

\begin{equation}\label{mu}
\Phi_\lambda(u)=c \quad \Longleftrightarrow \quad \lambda=\lambda_c(u):=\frac{I_1(u)-c}{I_2(u)}.
\end{equation}
Here we assume that $I_2(u) \neq 0$ for every $u \in X \setminus \{0\}$, so that $\lambda_c \in C^1(X \setminus \{0\})$.
Furthermore, a simple computation shows that   
\begin{equation}
\label{lp}\lambda_c'(u)=\frac{\Phi_{\lambda_c(u)}'(u)}{I_2(u)}, \quad \forall u \in X  \setminus \{0\}.
\end{equation}
Therefore \eqref{ef} can be reformulated as follows:
\begin{equation*}
\Phi_\lambda'(u)=0, \quad \Phi_{\lambda}(u)=c \quad \Longleftrightarrow \quad \lambda=\lambda_c(u),\quad \lambda_c'(u)=0,
\end{equation*}
i.e.  critical points (along with its associated critical values) of $\lambda_c$ yield (all) couples $(\lambda,u)$ solving \eqref{ef}. 

So far the issue of finding critical points (and critical values) of $\lambda_c$ has been tackled by means of Pohozaev's {\it fibering method}, which leads to the NGRQ functional, cf. \cite{I3,RSiS}. 
In contrast with \cite{RSiS}, we shall consider now the case where $I_1$ may contain nonhomogeneous terms. Our main model for this situation is the functional \begin{equation}\label{fm}
\Phi_\lambda(u)=\frac{1}{2}\int_{\Omega} |\nabla u|^2-\frac{\lambda}{2}\int_{\Omega} u^2-\int_{\Omega} F(u), \quad u \in H_0^1(\Omega),
\end{equation} 
which is associated to the boundary value problem
\begin{equation}\label{semi}
\begin{cases}
-\Delta u=\lambda u +f(u) &\mbox{ in } \Omega,\\
u=0 &\mbox{ on } \partial \Omega.
\end{cases}
\end{equation}
Here $\Omega \subset \mathbb{R}^N$ ($N \geq 1$) is a bounded domain, $F(s):=\int_0^s f(t) dt$, and $f\in C(\mathbb{R})$ has subcritical Sobolev growth. In this case the functional $\lambda_c$ is given by
 \[
\lambda_c(u)=\frac{\int_\Omega |\nabla u|^2-2\int_\Omega F(u)-2c}{\int_\Omega |u|^2},
\]
 for $u \in H_0^1(\Omega) \setminus \{0\}$.
 
 We shall prove the existence of a ground state and infinitely many pairs of critical points of $\lambda_c$ and relate it to the problem \eqref{ef}. These results shall be established under rather general conditions on $I_1$ and $I_2$, providing us with applications to a large variety of elliptic problems.
 
\subsection{Problems involving the  affine $p$-Laplacian}

The {\it affine} $p$-Laplacian operator is given by
\[
\Delta^{\cal A}_p u = -{\rm div} \left( H_{u}^{p-1}(\nabla u) \nabla H_{u}(\nabla u) \right)
\]
for every $u \in W^{1,p}_0(\Omega) \setminus \{0\}$. Here
 \[
H_{u}^p(\zeta) = \gamma_{N,p}^{-\frac Np}\; {\cal E}_{p, \Omega}^{N + p}(u) \int_{\mathbb{S}^{N-1}}  \| \nabla_\xi u\|_p^{-(N+p)}  |\langle \xi, \zeta \rangle|^p\, d\sigma(\xi)\ \ {\rm for}\ \zeta \in \mathbb{R}^N\, ,
\] $\gamma_{N,p} = \left( 2 \omega_{N+p-2} \right)^{-1} \left(N \omega_N \omega_{p-1}\right) \left(N \omega_N\right)^{p/N}$, $\nabla_\xi u(x)=\nabla u(x) \cdot \xi$, $\omega_k$ is the volume of the unit Euclidean ball in $\mathbb{R}^k$, and
\begin{equation} \label{dep}
{\mathcal{E}}_{p,\Omega}(u) = \gamma_{N,p} \left( \int_{\mathbb{S}^{N-1}} \| \nabla_\xi u\|_p^{-N}\, d\sigma(\xi)\right)^{-\frac{1}{N}}
\end{equation}
is the so-called {\it affine} $p$-energy. In addition, we set 
$\Delta^{\cal A}_p 0 =0$.

We say that $u \in W^{1,p}_0(\Omega) \setminus \{0\}$ is a nontrivial weak solution of \begin{equation} \label{P}
\left\{
\begin{array}{rlllr}
\Delta^{\cal A}_p u &=&  f(u) & {\rm in} & \Omega, \\
u&=&0 & {\rm on} & \partial \Omega,
\end{array}\right.
\end{equation} if 
\[
\langle \Delta^{\cal A}_p u, \varphi \rangle = \int_\Omega H_{u}^{p-1}(\nabla u) \nabla H_{u}(\nabla u) \cdot \nabla \varphi = \int_\Omega f(u) \varphi
\]
for every $\varphi \in W^{1,p}_0(\Omega)$. It follows that  nontrivial weak solutions of \eqref{P} correspond to
 critical points of the functional
\[
\Phi_{\cal A}(u) = \frac{1}{p} {\cal E}^p_{p,\Omega}(u) -  \int_\Omega F(u),  \quad u\in W^{1,p}_0(\Omega) \setminus \{0\},
\]
where $F(t) = \int_0^t f(s) ds$ for $t \in \mathbb{R}$, and $f\in C(\mathbb{R})$ satisfies some growth conditions.  By Theorem 1 of \cite{LM2}, $\Phi_{\cal A}$ is continuous in $W^{1,p}_0(\Omega)$ and $C^1$ in $W_0^{1,p}(\Omega) \setminus \{0\}$. Furthermore the geometry and compactness properties of $\Phi_{\cal A}$ are affected by the fact that ${\cal E}^p_{p,\Omega}$ is not convex and may remain bounded along unbounded sequences, so ${\cal E}^p_{p,\Omega}$ is not coercive (which strongly differs from the classical $p$-Laplacian setting). We refer the reader to \cite{LM2} for a discussion on these issues, as well as further properties of the {\it affine} $p$-Laplacian operator.

The existence of at least one nontrivial solution for  \eqref{P} has been studied recently. For a powerlike $f$ we refer to Corollary 3.8 in \cite{SchT}, Theorem 5.1 in \cite{ST}, and   Theorems 1.1-1.5 in \cite{LM1}, where direct minimization arguments are used. In \cite{LM2} the problem (\ref{P}) was treated for a $p$-superlinear $f$ satisfying the classical Ambrosetti-Rabinowitz condition (see Theorems 3 therein). To this end the authors adapt the mountain pass and minimization theorems for a functional $\Phi$ which is continuous on  $X$ and $C^1$ in $X \setminus \{0\}$. More recently the concave-convex case of the problem (\ref{P}) was treated in \cite{LM3} using the tools obtained from \cite{LM2} with an appropriate perturbation argument.

We shall observe that dealing with the Nehari manifold associated to $\Phi_{\cal A}$ yields a more direct approach (in comparison to \cite{LM2,LM3}) with regard to the Palais-Smale condition. Indeed, unlike \cite{LM2,LM3} we do not need to perturb the functional $\Phi_{\cal A}$ to obtain such property. In addition, Theorem \ref{tap} below seems to be the first result on the existence of infinitely many solutions for \eqref{P}. 

\subsection{Degenerate Kirchhoff type problems} Our third model of functional is given by \begin{eqnarray*}
\Phi(u) &=& \frac{1}{\theta} \left(\int_\Omega \vert \nabla u\vert^{p}\right)^{\frac{\theta}{p}} - \int_\Omega F(u)\\
&=& \frac{1}{\theta} \|u\|^{\theta} - \int_\Omega F(u),
 \quad u\in W^{1,p}_0(\Omega) \setminus \{0\},
\end{eqnarray*}
where $\theta\leq 1<p$ and $\theta \neq 0$ (in particular $\theta$ can be negative). Such functional is related to the degenerate Kirchhoff type problem

\begin{equation} \label{P0}
\left\{
\begin{array}{rlllr}
-\|u\|^{\theta-p} \Delta_p u &=&  f(u) & {\rm in} & \Omega, \\
u&=&0 & {\rm on} & \partial \Omega,
\end{array}\right.
\end{equation}
where $F(t) = \int_0^t f(s) ds$ for $t \in \mathbb{R}$, and $f\in C(\mathbb{R})$ is subcritical.
When $p=2$ the problem above reduces to the so-called Kirchhoff problem
\begin{equation} \label{k2}
\left\{
\begin{array}{rlllr}
-M(\|u\|^2) \Delta u &=&  f(u) & {\rm in} & \Omega, \\
u&=&0 & {\rm on} & \partial \Omega,
\end{array}\right.
\end{equation}
with $M(s^2)=s^{\theta-2}$. Since $M$ is not bounded away from zero for $\theta \leq 1$ the problem is said to be degenerate. This class of Kirchhoff type problems has been mostly studied for a nondegenerate $M$.  In the degenerate case we refer the reader to \cite{AA}, where \eqref{k2} is considered with $f(u)$ replaced by $\lambda f(x,u)$ and $M\in C^1([0,\infty))$ such that $M(t)\to 0$ as $t\to \infty$. Assuming $f(x,u)$ to be asymptotically linear in $u$ the authors use bifurcation arguments to prove the existence of a positive solution for $0<\lambda<\Lambda_1:=M(0)\lambda_1$, see Theorem 3.1 therein. In \cite{M} a positive solution is obtained for  \eqref{k2} with   $f(u)=\lambda u^{r-1}$, $1<r<2^*$, and $M$ such that $s^{2-r}M(s) \to \infty$ as $s \to 0$ and $s^{2-r}M(s) \to 0$ as $s \to \infty$, see  Example 3.1 therein.

Besides applying to \eqref{P0} in a degenerate case with respect to the left-hand side, our results hold for a wide class of right-hand sides, which allows in particular $f$ to vanish at infinity when $\theta<1$, see Example 1 in Section 6 below. More precisely, the more negative is $\theta-1$, the faster $f$ can vanish at infinity. \\
 
 
This article is organized as follows: in Section 2 we state and prove our main abstract result, Theorem \ref{tn}, and apply it to \eqref{ef}, which yields Theorem \ref{c1} and Corollary \ref{cc1}. Section 3 contains the proof of Theorem \ref{c1}, which is then applied to some boundary value problems in Section 4. We also consider some problems that do not fit in the conditions of Theorem \ref{c1} in Subsection 4.1. Finally, in Sections 5 and 6 we apply Theorem \ref{tn} to problems \eqref{P} and \eqref{P0}, respectively.
 
\subsection*{Notation} Throughout this article, we use the following notation:

\begin{itemize}
	\item Unless otherwise stated $\Omega$ denotes a bounded domain of $\mathbb{R}^N$ with $N\geq 1$.
	
	\item Given $r>1$, we denote by $\Vert\cdot\Vert_{r}$ (or $\Vert\cdot\Vert_{r,\Omega}$ in case we need to stress the dependence on $\Omega$) the usual norm in
	$L^{r}(\Omega)$, and by $r^*$ the critical Sobolev exponent, i.e. $r^*=\frac{Nr}{N-r}$ if $r<N$ and $r^*=\infty$ if $r \geq N$.
	
	\item Strong and weak convergences are denoted by $\rightarrow$ and
	$\rightharpoonup$, respectively.
	
	\item Given $g\in L^{1}(\Omega)$ we simply write $\int_{\Omega}g$ instead of $\int_{\Omega} g(x)\, dx$.

\end{itemize}

 \medskip
 \section{Main abstract results}
 \medskip
 
As mentioned above, we shall rely here on the Nehari manifold approach in the spirit of \cite{SW}, which we recall now.
Let $X$ be an infinite-dimensional uniformly convex Banach space, equipped with $\|\cdot \| \in C^1(X \setminus \{0\})$.
 We shall verify that the abstract setting of \cite{SW} applies to $\Phi$ although this functional is not defined at $u=0$.  To this end, we require the following condition:\\
 
 \begin{itemize}
 	\item[(H1)] For every $u \in X \setminus \{0\}$ the map $t \mapsto \Phi(tu)$, defined for $t>0$, has exactly one critical point $t(u)>0$, which is either a global maximum point or a global minimum point. Moreover, the map $u \mapsto t(u)$ is bounded away from zero in $\mathcal{S}$, and bounded from above in any compact subset of $\mathcal{S}$, where $\mathcal{S}$ is the unit sphere in $X$.\\
 \end{itemize}

 Let us recall that in \cite{SW} it is assumed that $\Phi \in C^1(X)$ with $\Phi(0)=0$, and $t(u)$ is a global maximum point of $t \mapsto \Phi(tu)$ for every $u \in X \setminus \{0\}$, cf. Figure \ref{fig1}. This assumption implies that $\Phi$ is positive (and therefore bounded from below) on $\mathcal{N}$. We shall deal with cases where $\Phi$ is not positive on $\mathcal{N}$, which may occur not only when $t(u)$ is a maximum point, but also when it is a minimum point. 
 
   \begin{figure}[h]
  	\centering
  	\begin{tikzpicture}[scale=0.8]
  	\draw[->] (-1,0) -- (4,0) node[below] {\scalebox{0.8}{$t$}};
  
  	\draw[->] (0,-1) -- (0,3) node[right] {\scalebox{0.8}{$\Phi(tu)$}};
  
  	\draw[blue,thick,domain=0:2.2,smooth,variable=\x] plot ({\x},{2*(\x)^2-(\x)^3});

  	\end{tikzpicture}
  		\caption{Plot of the map $t \mapsto \Phi(tu)$ under the conditions of \cite{SW}.} \label{fig1}
  \end{figure}
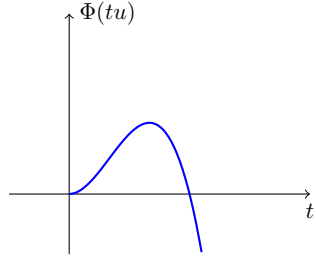

 \begin{figure}[h]
 	\centering
 	\begin{minipage}{0.2\linewidth} 
 	\begin{tikzpicture}[scale=0.8]
 \draw[->] (-1,0) -- (4,0) node[below] {\scalebox{0.8}{$t$}};
 
 \draw[->] (0,-1) -- (0,4) node[right] {\scalebox{0.8}{$\Phi(tu)$}};
 
 \draw[blue,thick,domain=0.2:3.8,smooth,variable=\x] plot ({\x},{(\x-2)^2+.5});	
 
 \end{tikzpicture}

\end{minipage}
 \hspace{.1\linewidth}
	\begin{minipage}{0.2\linewidth} 
	\begin{tikzpicture}[scale=0.8]
\draw[->] (-1,0) -- (4,0) node[below] {\scalebox{0.8}{$t$}};

\draw[->] (0,-3) -- (0,1) node[right] {\scalebox{0.8}{$\Phi(tu)$}};

\draw[blue,thick,domain=0.4:3.6,smooth,variable=\x] plot ({\x},{-(\x-2)^2-.5});	

\end{tikzpicture}

\end{minipage}
 \hspace{.1\linewidth}
\begin{minipage}{0.2\textwidth} 
	\begin{tikzpicture}[scale=0.8]
	\draw[->] (-1,0) -- (4,0) node[below] {\scalebox{0.8}{$t$}};

	\draw[->] (0,-1) -- (0,4) node[right] {\scalebox{0.8}{$\Phi(tu)$}};

	\draw[blue,thick,domain=0.2:4,smooth,variable=\x] plot ({\x},{(\x)^(-1)+4*(\x)^2*((\x)^2+1)^(-1)-2.1});	

	\draw [thick,dashed, red] (0,2) -- (4,2);

	\end{tikzpicture}
\end{minipage}
\caption{Some possible behaviors of the map $t \mapsto \Phi(tu)$ under (H1).} \label{fig2}
\end{figure}
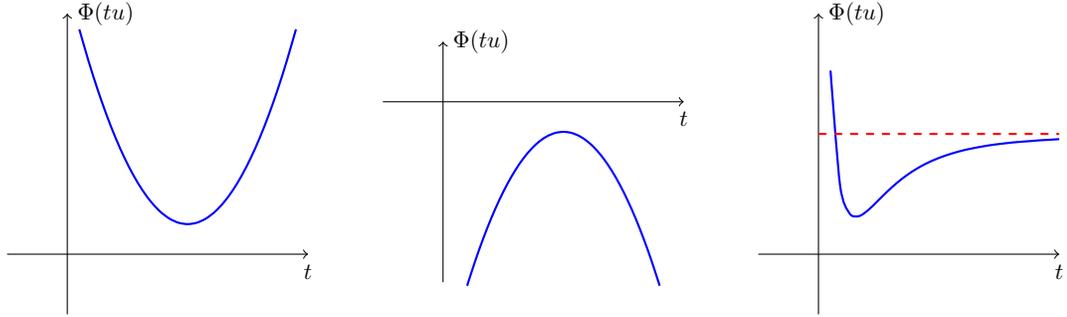


From (H1) it follows that $$\mathcal{N}=\{t(u)u \in X \setminus \{0\}: u \in \mathcal{S} \}.$$ In addition,  the uniqueness of $t(u)$ and the behavior of the map $u \mapsto t(u)$ on $\mathcal{S}$ (which is already assumed in condition $(A3)$ of \cite{SW}) entails that $\mathcal{N}$ is homeomorphic to $\mathcal{S}$, and there is a one-to-one correspondence between critical points of $\Phi$ and critical points of the map $u \mapsto \Phi(t(u)u)$, $u \in \mathcal{S}$. As a matter of fact, let $\delta>0$ be such that $t(u) \ge \delta$ for any $u \in \mathcal{S}$. Then the behavior of $t \mapsto \Phi(tu)$ for $t \in (0,\delta)$ and  $u \in \mathcal{S}$ does not affect the above one-to-one correspondence as this map has no critical points in $(0,\delta)$ for any $u \in \mathcal{S}$. So in particular we may have $\Phi(tu)\to \infty$ or $-\infty$ as $t \to 0^+$, cf. Figure \ref{fig2}.

The previous features imply that under (H1) the functional $\Phi$ has a ground state critical point whenever it achieves its infimum on $\mathcal{N}$.
The following standard condition shall then provide us with infinitely many pairs of critical points of $\Phi$:\\

 \begin{itemize}
	\item[(H2)]  $\Phi$ is even, bounded from below on $\mathcal{N}$, and satisfies the Palais-Smale condition on $\mathcal{N}$, i.e.
 	any sequence $(u_n) \subset \mathcal{N}$ such that $(\Phi(u_n))$ is bounded and $\Phi'(u_n) \to 0$ has a convergent subsequence (in $X \setminus \{0\}$).\\
 \end{itemize}

Indeed, the latter condition enables the use of the Ljusternik-Schnirelman principle.
Let $$\lambda_n:=\displaystyle
\inf_{F\in \mathcal{F}_n}\sup_{u\in F}\Phi(t(u)u),$$ where $
\mathcal{F}_n:=\{F\subset \mathcal{S}: F \mbox{ is compact, symmetric, and } \gamma(F)\ge n\}$, and $\gamma(F)$ is the Krasnoselskii genus of $F$. As stated in  \cite[Theorem 2]{SW}, every $\lambda_n$ is a critical value of $\Phi$, and if the nondecreasing sequence $(\lambda_n)$ is stationary, i.e. $\lambda_k=\lambda_{k+j}$ for some $k \in \mathbb{N}$ and every $j \in \mathbb{N}$, then $\lambda_k$ has infinitely many associated pairs of critical points.

Finally, the following condition shall provide us with infinitely many critical values of $\Phi$:\\

\begin{itemize}
	\item[(H3)] $\Phi(t(u_n)u_n) \to \infty$ for any $(u_n) \subset \mathcal{S}$ such that $u_n \rightharpoonup 0$ in $X$.\\ 
\end{itemize}

From the definition of $\lambda_n$ it is clear that $$\lambda_1=\displaystyle \inf_{u \in \mathcal{S}} \Phi(t(u)u)=\inf_{u \in \mathcal{N} } \Phi(u).$$ 

The previous discussion is summarized in the following result, which is proved by checking that the arguments of \cite{SW} still apply to our situation:

\begin{theorem}\label{tn}
Under (H1) the following assertions hold: 
\begin{enumerate}
\item If $\lambda_1$ is achieved then it is the ground state level of $\Phi$, i.e. the least critical level of $\Phi$.
\item If (H2) holds then 
$(\lambda_n)$ is a nondecreasing sequence of critical values of $\Phi$. In particular, this functional has infinitely many pairs of critical points.
\item  If (H2) and (H3) hold, then $\lambda_n \to \infty$ as $n \to \infty$, so that $\Phi$ has infinitely many critical values.
\end{enumerate}
\end{theorem}

\begin{proof}
It suffices to see that  \cite[Corollary 10]{SW} still holds if $\Phi \in C^1(X \setminus \{0\})$ satisfies (H1), as this result does not take into account the value $\Phi(0)$ (note that the proofs of \cite[Propositions 8 and 9]{SW}  hold in a similar way if $t(u)$ is a minimum point of $t \mapsto \Phi(tu)$). In particular, if $\lambda_1$ is achieved then it is the ground state level of $\Phi$. Moreover, Theorem 2 in \cite{SW} yields that $(\lambda_n)$ is a sequence of critical values of $\Phi$, which has infinitely many pairs of critical points. Finally, assertion (3) follows from Lemma A.1 in \cite{RSS}.
\end{proof}


\begin{remark}
Instead of assuming the `global' Palais-Smale condition on $\mathcal{N}$, one may assume that
the Palais-Smale condition on $\mathcal{N}$ holds at the levels $\lambda_k$. So $\lambda_k$ is a critical value of $\Phi$ if (H2) is replaced by:\\
 \begin{itemize}
	\item[(H2')]  $\Phi$ is even, bounded from below on $\mathcal{N}$, and satisfies the Palais-Smale condition on $\mathcal{N}$ at the level $\lambda_k$, i.e.
	any sequence $(u_n) \subset \mathcal{N}$ such that $\Phi(u_n)\to \lambda_k$ and $\Phi'(u_n) \to 0$ has a convergent subsequence (in $X \setminus \{0\}$).
\end{itemize}	

This condition is particularly useful when one looks for prescribed energy solutions of the Br\'ezis-Nirenberg problem, which is considered in Subsection 4.1 below.
\end{remark} 

\subsection{Applications to the prescribed energy problem}

Let us consider the functional $\lambda_c$ given by \eqref{mu} under the assumptions of Subsection \ref{sspep}.
We assume that $I_2$ is $\alpha$-homogeneous for some $\alpha>1$, i.e. $I_2(tu)=t^{\alpha}I_2(u)$ for any $u \in X \setminus \{0\}$ and $t>0$. Note that in  \eqref{fm} we have $I_2(u)=\frac{1}{2}\int_\Omega u^2$, which is a $2$-homogeneous functional.
 
From \eqref{lp} we see that $u \in \mathcal{N}(\lambda_c)$ if, and only if, $\Phi_{\lambda_c(u)}'(u)u=0$, i.e. $I_1'(u)u-\lambda_c(u)I_2'(u)u=0$.
Since $I_2$ is $\alpha$-homogeneous we have $I_2'(u)u=\alpha I_2(u)$, and consequently
\[
\mathcal{N}(\lambda_c)=\{u \in X  \setminus \{0\}: H(u)+\alpha c=0\},
\]
where
$$H(u):=I_1'(u)u-\alpha I_1(u), \quad u \in X.$$
We see that $H \in C^1(X)$ and $H(0)=0$ if $I_1(0)=0$, which shall be assumed in the sequel.
The behavior of the map $t \mapsto H(tu)$ is then crucial to understand $\mathcal{N}(\lambda_c)$. In this article we shall focus on the case where the  map  $t \mapsto \lambda_c(tu)$ has a unique critical point for any $u\neq 0$, which occurs if one of the following conditions is satisfied:\\

\begin{itemize}
	\item[(F1)] For any $u \in X \setminus \{0\}$ there exists $s(u) \ge 0$ such that the map $t \mapsto H(tu)$ is increasing in $(0,s(u))$ and decreasing in $(s(u),\infty)$. Moreover $\displaystyle \lim_{t \to \infty} H(tu)=-\infty$ uniformly on weakly compact subsets of $X \setminus \{0\}$.\\ 
   \item[(F2)]   For any $u \in X \setminus \{0\}$ there exists $s(u) \ge 0$ such that the map $t \mapsto H(tu)$ is decreasing in $(0,s(u))$ and increasing in $(s(u),\infty)$. Moreover $\displaystyle \lim_{t \to \infty} H(tu)=\infty$ uniformly on weakly compact subsets of $X \setminus \{0\}$.\\
\end{itemize}

Note that these conditions include the case $s(u)=0$ (i.e.  the map $t \mapsto H(tu)$ is monotone), which occurs in several problems.
Our main results on \eqref{ef} will be stated under the condition that $H$ is weakly semicontinuous, which is the case, for instance, in  \eqref{semi} with $f$  continuous and having subcritical Sobolev growth, as $H(u)=\int_\Omega \left(2F(u)-f(u)u\right)$ for $u \in H_0^1(\Omega)$, so that $H$ is weakly continuous. Note also that in this case we have $H(u)=\int_\Omega h(u)$ with $h(s)=2F(s)-f(s)s$, which appears in the litterature related to superlinear type problems without the Ambrosetti-Rabinowitz condition. We refer to \cite{FS,L} for more on this subject.

The following additional condition shall be imposed to have (H2) satisfied:\\

\begin{itemize}
	\item[(F3)] $I_1=J-K$, where $J,K \in C^1(X)$ are such that $K'$ is completely continuous, and there exist $C_1,C_2>0$, $\beta \geq\alpha$, and $\eta>1$ such that $J(u)\geq C_1\|u\|^{\beta}$ and $\left(J'(u)-J'(v)\right)(u-v)\geq C_2(\|u\|^{\eta-1}-\|v\|^{\eta-1})(\|u\|-\|v\|)$ for any $u,v \in X$. Moreover,  if $(F2)$ holds then $\displaystyle \limsup_{\|u\|\to \infty} \frac{K(u)}{\|u\|^{\alpha}}\le 0$.\\
\end{itemize}

Note that in \eqref{semi} we have $I_1=J-K$ with $J(u)=\frac{1}{2}\|u\|^2$ and $K(u)=\int_\Omega F(u)$, for $u \in H_0^1(\Omega)$.

We are now in position to state our main result on \eqref{ef}:
\begin{theorem}\label{c1}
Let $c \in \mathbb{R}$ and $I_1,I_2 \in C^1(X)$ be such that $I_1(0)=I_2(0)=0$, $I_2(u)>0$ for $u \in  X \setminus \{0\}$,  $I_2$ is $\alpha$-homogeneous for some $\alpha>1$, and $I_2'$ is completely continuous. Assume in addition  $(F1)$ (respect. $(F2)$) if $c>0$ (respect. $c<0$), $(F3)$, and $H$ is weakly lower (respect. upper) semicontinuous. Then:
\begin{enumerate}
\item $\lambda_c$  has a ground state level, given by $\lambda_{1,c}:=\displaystyle \min_{u \in \mathcal{S}} \lambda_c(t_c(u)u)=\min_{u \in \mathcal{N}(\lambda_c) } \lambda_c(u)$. 
\item If, in addition, $I_1,I_2$ are even then $\lambda_c$ has a nondecreasing unbounded sequence of critical values, given by \begin{equation*}\label{lnc}\lambda_{n,c}:=\displaystyle\inf_{F\in \mathcal{F}_n}\sup_{u\in F}\lambda_c(t_c(u)u).\end{equation*}
 In particular $\lambda_c$ has infinitely many pairs of critical points.
\end{enumerate}
\end{theorem}

\begin{remark}
Instead of assuming in $(F3)$ that $\left(J'(u)-J'(v)\right)(u-v)\geq C_2(\|u\|^{\eta-1}-\|v\|^{\eta-1})(\|u\|-\|v\|)$ for any $u,v \in X$, one may assume that $J$ is weakly lower semicontinuous and $H$ is weakly continuous to show that the ground state level of $\lambda_c$ is achieved. Indeed, in this case $\mathcal{N}(\lambda_c)$ is weakly sequentially closed and $\lambda_c$ is weakly lower semicontinuous in $X \setminus \{0\}$, and since $0$ does not belong to the weak sequential closure of $\mathcal{N}(\lambda_c)$, we deduce that $\lambda_c$ achieves it infimum over $\mathcal{N}(\lambda_c)$.
\end{remark}

Theorem \ref{c1} takes the following form when stated in terms of the functional $\Phi_\lambda$:
\begin{corollary}\label{cc1}
Under the assumptions of Theorem \ref{c1}, the following assertions hold:
\begin{enumerate}
\item $\Phi_\lambda$ has a critical point $u_{1,c}$ at the level $c$ for $\lambda=\lambda_{1,c}$, and has no such critical point for $\lambda<\lambda_{1,c}$.
\item If, in addition, $I_1,I_2$ are even, then for every $n\in \mathbb{N}$ the functional $\Phi_\lambda$ has a pair of critical points $\pm u_{n,c}$ at the level $c$ for $\lambda=\lambda_{n,c}$.
Moreover $\lambda_{n,c} \nearrow \infty$ as $n \to \infty$, so there exist infinitely many couples $(\lambda, \pm u)$ such that $\Phi_\lambda'(\pm u)=0$ and $\Phi_\lambda(\pm u)=c$.
\end{enumerate}
\end{corollary}

\begin{remark}
When $X$ is a function space (which is the case in our applications) and $I_1,I_2$ are even, one can choose $u_{1,c}$ to be nonnegative in Corollary \ref{cc1}.
\end{remark}

\medskip
\section{Proof of Theorem \ref{c1}}
\medskip

Let us show now that Theorem \ref{c1} follows from Theorem \ref{tn}. To this end we shall prove in the next lemmas that $\lambda_c$ satisfies (H1), (H2), and (H3). 

\begin{lemma}\label{l0}
Let $c \neq 0$. Assume that $I_2$ is $\alpha$-homogeneous and $I_1$ satisfies $(F1)$ (respect. $(F2)$) if $c>0$ (respect. $c<0$). Then:
\begin{enumerate}
\item  For every $u \in X \setminus \{0\}$ the map $t \mapsto \lambda_c(tu)$ has exactly one critical point $t_c(u)>0$, which is a global maximum point (respect. minimum point). 
\item  If $(u_n) \subset \mathcal{S}$ and $t_c(u_n) \to \infty$ then $u_n \rightharpoonup 0$ in $X$. In particular, $t_c$ is bounded from above in any compact subset of $\mathcal{S}$.

\item $\lambda_c$ satisfies (H1).
\end{enumerate}
\end{lemma}

\begin{proof}\strut
\begin{enumerate}
\item Let $u \in X \setminus \{0\}$. By homogeneity we have $\lambda_c(tu):=\frac{I_1(tu)-c}{t^{\alpha}I_2(u)}$, so that 
\begin{eqnarray*}
\frac{d}{dt} \lambda_c(tu)&=&\frac{t^{\alpha}I_2(u)I_1'(tu)u-\alpha t^{\alpha-1}I_2(u)(I_1(tu)-c)}{t^{2\alpha}I_2(u)^2}=\frac{I_1'(tu)tu-\alpha(I_1(tu)-c)}{t^{\alpha+1}I_2(u)}\\
&=&\frac{H(tu)+\alpha c}{t^{\alpha+1}I_2(u)},
\end{eqnarray*}
i.e.  the map $t \mapsto \lambda(tu)$ has a critical point if, and only if, 
\begin{equation*}\label{et}
H(tu)=-\alpha c.
\end{equation*}
Since $H(0)=0\gtrless -\alpha c$, we see that under $(F1)$ (respect. $(F2)$) the map $t \mapsto \lambda(tu)$ has a unique critical point $t_c(u)>0$, which is a global maximum (respect. minimum) point if $c>0$ (respect. $c<0$). \\

\item Indeed, if $t_c(u_n) \to \infty$  and $u_n \rightharpoonup u \neq 0$ in $X$, then by $(F1)$ (respect. $(F2)$) we have $|H(t_c(u_n) u_n)| \to \infty$, which contradicts $H(t_c(u_n) u_n)=-\alpha c$. \\

\item It remains to show that $t_c$ is away from zero on $\mathcal{S}$, which follows from the fact that $H$ is continuous and $H(0)=0$. 
\end{enumerate}	
\end{proof}

\subsection{On the Palais-Smale condition for $\lambda_c$}
Let us first show that \eqref{lp} provides an easy way to characterize the Palais-Smale condition for $\lambda_c$ in terms of the functional $\Phi_\lambda$. Recall that $(u_n) \subset X  \setminus \{0\}$ is a Palais-Smale sequence for $\lambda_c$ at the level $d$ if $\lambda_c(u_n) \to d$ and $\lambda_c'(u_n) \to 0$ in $X^*$.

\begin{lemma}\label{l00}
Let $(u_n) \subset X  \setminus \{0\}$ be a Palais-Smale sequence for $\lambda_c$ at the level $d$, i.e. $\lambda_c(u_n) \to d$ and $\lambda_c'(u_n) \to 0$ in $X^*$.
 If $I_2'(u_n)$ and $I_2(u_n)$ are bounded then $(u_n)$ is a Palais-Smale sequence for $\Phi_d$ at the level $c$, i.e. $\Phi_d(u_n) \to c$ and $\Phi_d'(u_n) \to 0$ in $X^*$.
\end{lemma}

\begin{proof}
	From the definition of $\lambda_c$ we have that $\Phi_{\lambda_c(u_n)}(u_n)=c$, whereas
	\eqref{lp} yields that $\Phi_{\lambda_c(u_n)}'(u_n)=I_2(u_n) \lambda_c'(u_n)$. Thus
	$$\Phi_d(u_n)=\Phi_{\lambda_c(u_n)}(u_n) +(\lambda_c(u_n)-d)I_2(u_n)=c+(\lambda_c(u_n)-d)I_2(u_n)$$
	and $$\Phi_d'(u_n)=\Phi_{\lambda_c(u_n)}'(u_n) +(\lambda_c(u_n)-d)I_2'(u_n)=I_2(u_n) \lambda_c'(u_n)+(\lambda_c(u_n)-d)I_2'(u_n).$$
	The boundedness of $I_2'(u_n)$ and $I_2(u_n)$ yields the desired conclusion.
\end{proof}

The boundedness condition on $I_2'(u_n)$ and $I_2(u_n)$ assumed in Lemma \ref{l00} is clearly satisfied if $(u_n)$ is bounded and $I_2'$ is completely continuous, which provides the following result:

\begin{corollary}
Assume that $I_2'$ is completely continuous. Then 
any bounded Palais-Smale sequence for $\lambda_c$ at the level $d$ is a Palais-Smale sequence for $\Phi_d$ at the level $c$. In particular, if any Palais-Smale sequence for $\lambda_c$ at the level $d$ is bounded, and $\Phi_d$ satisfies the Palais-Smale condition at the level $c$, then $\lambda_c$ satisfies the Palais-Smale condition at the level $d$.
\end{corollary}

The proof of Lemma \ref{l00} shows that this result can be extended as follows:

\begin{lemma}
Let $(u_n) \subset X  \setminus \{0\}$ and $(c_n) \subset \mathbb{R}$ be such that $\lambda_n=\lambda_{c_n}(u_n)\to d$, $\lambda_{c_n}'(u_n) \to 0$ in $X^*$, and $c_n \to c$. If in addition $I_2'(u_n)$ and $I_2(u_n)$ are bounded then $(u_n)$ is a Palais-Smale sequence for $\Phi_d$ at the level $c$.
\end{lemma}

Let us show now that $\lambda_c$ satisfies (H2) and (H3):
\begin{lemma}\label{l1}

 Let  $c \neq 0$. Assume that $I_2$ is $\alpha$-homogeneous and weakly continuous,  $I_1$ satisfies $(F1)$ (respect. $(F2)$), and $H$ is weakly lower (respect. upper) semicontinuous  if $c>0$ (respect. $c<0$).	Assume in addition that $(F3)$ holds.
\begin{enumerate}
\item If $(u_n) \subset \mathcal{N}(\lambda_c)$ and $u_n \rightharpoonup u$ then $u \neq 0$, i.e. $0$ does not belong to the weak sequential closure of $\mathcal{N}(\lambda_c)$.
\item $\lambda_c$ is coercive and bounded from below on $\mathcal{N}(\lambda_c)$.
\item  Let  $(u_n) \subset \mathcal{S}$. Then
$$u_n \rightharpoonup 0 \mbox{ in } X \quad \Leftrightarrow \quad t_c(u_n) \to \infty \quad \Leftrightarrow \quad \lambda_c(t_c(u_n)u_n) \to \infty.$$
In particular, $\lambda_c$ satisfies (H3).
\item $\lambda_c$ satisfies the Palais-Smale condition on $\mathcal{N}$.

\end{enumerate}	
\end{lemma}

\begin{proof}\strut
\begin{enumerate}
\item Let $c>0$. 
If $(u_n) \subset \mathcal{N}(\lambda_c)$ and $u_n \rightharpoonup u$ then the weak lower semicontinuity of $H$ implies that $H(u) \le \liminf H(u_n)=-\alpha c < 0$, i.e. $u \neq 0$. A similar argument applies if $c<0$.\\

\item Let $(u_n) \subset \mathcal{N}(\lambda_c)$ be such that $\|u_n\| \to \infty$. Then  $v_n=\frac{u_n}{\|u_n\|}\rightharpoonup 0$ in $X$. Indeed, otherwise we have (up to a subsequence) $v_n \rightharpoonup v \not \equiv 0$ in $X$, and consequently by $(F1)$ or $(F2)$ we have  $\alpha |c|=|H(u_n)|=|H(\|u_n\|v_n)| \to \infty$,  which yields a contradiction. Therefore $v_n \rightharpoonup 0$ in $X$.

If $(F1)$ holds and $c>0$ then, since $t=1$ is now a global maximum point of $t \mapsto \lambda_c(tu_n)$, we have
$$\lambda_c(u_n) \ge \lambda_c(tv_n) \geq \frac{C_1t^{\beta-\alpha}-\frac{K(tv_n)+c}{t^{\alpha}}}{I_2(v_n)}, \quad \forall t>0.$$
Chosing $t>0$ such that $C_1t^{\beta-\alpha}-ct^{-\alpha}>0$ and using the fact that $K(tv_n) \to 0$ and $I_2(v_n) \to 0$, we find that $\lambda_c(u_n) \to \infty$. 
If now $(F2)$ holds and $c<0$ then, by $(F3)$, we have $\liminf \frac{I_1(u_n)}{\|u_n\|^{\alpha}}>0$, so that
$$\lambda_c(u_n) \ge \frac{I_1(u_n)}{I_2(u_n)}=\frac{I_1(u_n)}{\|u_n\|^{\alpha}I_2(v_n)}\to \infty.$$
Thus $\lambda_c(u_n) \to \infty$, so $\lambda_c$ is coercive  on $\mathcal{N}(\lambda_c)$. Let us show that it is also bounded from below therein.
 Assume by contradiction that $(u_n) \subset \mathcal{N}(\lambda_c)$ and $\lambda_c(u_n)\to -\infty$. We know that $(u_n)$ is bounded, so that up to a subsequence $u_n \rightharpoonup u$ in $X$. Since $(F3)$ implies that $(I_1(u_n))$ is bounded from below, we must have $I_2(u_n) \to 0$, i.e. $u_n \rightharpoonup 0$ in $X$, which contradicts the previous item. Therefore $\lambda_c$ is bounded from below on $\mathcal{N}(\lambda_c)$.\\

\item If $u_n \rightharpoonup 0$ in $X$ with $(t_c(u_n))$ bounded, then $t_c(u_n)u_n \rightharpoonup 0$ in $X$, which contradicts (1). Thus $t_c(u_n) \to \infty$, and arguing as in the previous item we see that $\lambda_c(t_c(u_n)u_n) \to \infty$.
Now, if $\lambda_c(t_c(u_n)u_n) \to \infty$ then, by (1) we can not have $t_c(u_n)u_n \rightharpoonup 0$ in $X$. Since $(I_1(t_c(u_n)u_n))$ is bounded from below if $(t_c(u_n)u_n)$ is bounded, we infer that $(t_c(u_n)u_n)$ is unbounded, i.e. $t_c(u_n) \to \infty$. Lemma \ref{l0} yields that $u_n \rightharpoonup 0$ in $X$.\\

\item First of all observe that since $I_2$ is weakly continuous, it is bounded on bounded sets of $X$. Let $(u_n) \subset \mathcal{N}(\lambda_c)$ be a Palais-Smale sequence for $\lambda_c$. By (2) we know that $(u_n)$ is bounded, so that up to a subsequence $u_n \rightharpoonup u \neq 0$ in $X$. From \eqref{lp} we know that $\Phi_{\lambda_c(u_n)}'(u_n)=I_2(u_n)\lambda_c'(u_n)$, and since $(I_2(u_n))$ is bounded we deduce that $\Phi_{\lambda_c(u_n)}'(u_n) \to 0$. Thus, from the boundedness of $(\lambda_c(u_n))$ and the complete continuity of $I_2'$ and $K'$,  we obtain  that $J'(u_n)(u_n-u) \to 0$, and consequently 
$\left(J'(u_n)-J'(u)\right)(u_n-u) \to 0$. Using the inequality $\left(J'(u)-J'(v)\right)(u-v)\geq C_2(\|u\|^{\eta-1}-\|v\|^{\eta-1})(\|u\|-\|v\|)$ we find that $\|u_n\| \to \|u\|$, and the uniform convexity of $X$ provides the desired conclusion.
\end{enumerate}	
\end{proof}

\begin{remark}\label{r1}
Under $(F3)$ assume in addition that the map $u \mapsto J'(u)u-\alpha J(u)$ is weakly continuous. Since $K'$ is completely continuous, it follows that $K$ and $u \mapsto K'(u)u$ are weakly continuous, and so is $H$. Note that if $J$ is $\alpha$-homogeneous then $J'(u)u-\alpha J(u)=0$ for every $u$, so in this case $H$ is weakly continuous under $(F3)$.
\end{remark}


\medskip
\section{Applications of Theorem \ref{c1} to some boundary value problems}
\medskip

Next we apply Theorem \ref{c1} (more precisely Corollary \ref{cc1}) to the prescribed energy problem associated to some boundary value problems. First we deal with some cases where 
$H$ is weakly semicontinuous.

\begin{corollary}\label{c2}
Let  $c>0$ (respect. $c<0$) and $f\in C(\mathbb{R})$ satisfy
\begin{itemize}
\item[(f1)] 	$|f(s)| \leq C(1+|s|^{r-1})$ for some $C>0$ and $1<r<2^*$, and every $s \in \mathbb{R}$.
\item[(f2)]  $s \mapsto G(s):= sf(s)-2F(s)$ is increasing (respect. decreasing) in $(0,\infty)$, and decreasing (respect. increasing) in $(-\infty,0)$, and $\displaystyle \lim_{|s| \to \infty}G(s)=\infty$ (respect. $-\infty$). 
\end{itemize}
 Then:
\begin{enumerate}
\item The problem \eqref{semi} has a  nontrivial solution $ u_{1,c}$ having energy $c$ for $\lambda=\lambda_{1,c}$, and has no such solution for $\lambda<\lambda_{1,c}$.\\ 
\item If in addition $f$ is odd, then we can choose $u_{1,c}$ to be nonnegative, and for every $n\in \mathbb{N}$ the problem \eqref{semi} has a pair of nontrivial solutions $\pm u_{n,c}$ having energy $c$ for $\lambda=\lambda_{n,c}$. Moreover $\lambda_{n,c} \nearrow \infty$ as $n \to \infty$.
\end{enumerate}
\end{corollary}

\begin{proof}
Let $I_1,I_2$ be given by $I_1(u)=\frac{1}{2}\int_\Omega |\nabla u|^2-\int_\Omega F(u)$ and $I_2(u)=\frac{1}{2}\int_\Omega u^2$ for $u \in H_0^1(\Omega)$. Then $I_2$ is $2$-homogeneous and $I_2'$ clearly is completely continuous. Moreover $I_1$ satisfies $(F3)$ with $J(u)=\frac{1}{2}\int_\Omega |\nabla u|^2$ and $K(u)=\int_\Omega F(u)$. Note in particular that $H(u)=-\int_\Omega G(u)$ satisfies $(F1)$ (respect. $(F2)$) if $c>0$ (respect. $c<0$), and it is weakly continuous by Remark \ref{r1}.
\end{proof}

The condition $(f2)$ above can be replaced by the following one, which is standard when applying the Nehari manifold method to \eqref{semi}, cf. \cite[Theorem 16]{SW}:
\begin{itemize}
	\item[(f2')]  $s \mapsto \frac{f(s)}{|s|}$ is increasing (respect. decreasing) in $\mathbb{R} \setminus \{0\}$, and $\displaystyle \lim_{|s| \to \infty} \frac{f(s)}{s}=\infty$ (respect. $-\infty$).
\end{itemize}
Indeed, this condition is slightly stronger than $(f2)$, see e.g. \cite{FS,L}.
We deduce then the following result:

\begin{corollary}\label{c3}
Let  $c>0$ (respect. $c<0$) and $f\in C(\mathbb{R})$ satisfy (f1) and (f2'). Then the conclusions of Corollary \ref{c2} remain valid.
\end{corollary}

\begin{remark}
Note that $(f2')$ implies that $\displaystyle \lim_{s \to 0^{\pm}}  \frac{f(s)}{s}<\infty$ (respect. $>-\infty$). So at zero $f$ may be superlinear ($\displaystyle \lim_{s \to 0}  \frac{f(s)}{s}=0$), asymptotically linear ($\displaystyle \lim_{s \to 0}  \frac{f(s)}{s}=l \in \mathbb{R} \setminus \{0\}$), or satisfy $\displaystyle \lim_{s \to 0}  \frac{f(s)}{s}=-\infty$ (respect. $\infty$).

\end{remark}

\begin{remark}
Corollary \ref{c2} has a natural extension to the $p$-Laplacian problem
\[
\begin{cases}
-\Delta_p u=\lambda |u|^{p-2}u +f(u) &\mbox{ in } \Omega,\\
u=0 &\mbox{ on } \partial \Omega,
\end{cases}
\]
if we assume that $f$ satisfies (f1) and (f2) with $2$ replaced by $p$, and 
 \begin{equation*}
\Phi_\lambda(u)=\frac{1}{p}\int_{\Omega} |\nabla u|^p-\frac{\lambda}{p}\int_{\Omega} |u|^p-\int_{\Omega} F(u), \quad u \in W_0^{1,p}(\Omega).
\end{equation*}
We may also deduce the analogue of Corollary \ref{c3}, assuming now (f2') with 
$\frac{f(s)}{|s|}$ replaced by $\frac{f(s)}{|s|^{p-1}}$ and $\frac{f(s)}{s}$ by $ \frac{f(s)}{|s|^{p-2}s}$.
\end{remark}

Let us consider now the problem \begin{equation}\label{cc}
\begin{cases}
-\Delta u=\lambda |u|^{q-2}u +f(u) &\mbox{ in } \Omega,\\
u=0 &\mbox{ on } \partial \Omega,
\end{cases}
\end{equation}
with $1<q<2$ and $f\in C^1(\mathbb{R})$ being subcritical and superlinear, in which case the problem has a concave-convex nature.

We have now  $I_1(u)=\frac{1}{2}\int_\Omega |\nabla u|^2-\int_\Omega F(u)$ and $I_2(u)=\frac{1}{q}\int_\Omega |u|^q$ for $u \in H_0^1(\Omega)$, so that
\begin{equation}\label{lex}
\lambda_c(u)=\frac{\frac{q}{2}\int_\Omega |\nabla u|^2-q\int_\Omega F(u)-qc}{\int_\Omega |u|^q},
\end{equation}
for $u \in H_0^1(\Omega) \setminus \{0\}$.
Thus $I_2$ is now $q$-homogeneous, and
$$H(u)=\frac{2-q}{2}\int_\Omega |\nabla u|^2+\int_\Omega \left(qF(u)-f(u)u\right).$$
Thus $H$ is weakly lower semicontinuous.

We shall assume that 

	\begin{enumerate}
		\item [(f3)] $t\mapsto \frac{(q-1)f(t)}{t}-f'(t)$ is decreasing in $(0,\infty)$ and increasing in $(-\infty,0)$, and $$\displaystyle \lim_{|t|\to \infty} \left[\frac{(q-1)f(t)}{t}-f'(t)\right]=-\infty.$$\\
	\end{enumerate} 

This condition is clearly satisfied by $f(t)=|t|^{p-2}t$, as well as $f(t)=|t|^{p-2}t \ln (|t|+1)$ and $f(t)=\frac{|t|^{p-1}t}{|t|+1} $ with $p \in (2,2^*)$.
\begin{corollary}
	Let  $c>0$, $1<q<2$, and $f\in C^1(\mathbb{R})$ with $f'(0)=f(0)=0$. Assume in addition (f1) and (f3). Then the conclusions of Corollary \ref{c2} remain valid for \eqref{cc}.
\end{corollary}

\begin{proof}
Let us start showing that $(F1)$ is satisfied. We have
$$H(tu)=\frac{2-q}{q}t^2\int_\Omega |\nabla u|^2+\int_\Omega \left(qF(tu)-f(tu)tu\right),$$
so $$\frac{d}{dt} H(tu)=t\left[ \frac{(2-q)2}{q}\int_\Omega |\nabla u|^2+\int_\Omega \left(\frac{(q-1)f(tu)tu-f'(tu)(tu)^2}{t^2}\right)\right].$$
Note also that $$\int_\Omega \left(\frac{(q-1)f(tu)tu-f'(tu)(tu)^2}{t^2}\right)= \int_{u \ne 0} \left((q-1)\frac{f(tu)}{tu}-f'(tu)\right)u^2.$$
Moreover, from $f'(0)=f(0)=0$ we have $\frac{f(s)}{s}\to 0$ as $s \to 0$. 
By (f3) we infer that the map $t \mapsto \int_\Omega \left(\frac{(q-1)f(tu)tu-f'(tu)(tu)^2}{t^2}\right)$ is decreasing and goes to $0$ (respect $-\infty$) as $t \to 0$ (respect. $t \to \infty$). Finally, since (f3) yields that $\frac{qF(s)-f(s)s}{s^2} \to -\infty$ as $|s| \to \infty$, we deduce that $$\int_\Omega \frac{qF(tu)-f(tu)tu}{t^2} \to -\infty \quad \mbox{as} \quad t \to \infty$$ uniformly on weaky compact subsets of $X \setminus \{0\}$, i.e. $(F_1)$ is satisfied. Note also that $(F3)$ holds with $J(u)=\frac{1}{2}\int_\Omega |\nabla u|^2$ and $K(u)=\int_\Omega F(u)$.
\end{proof}

\subsection{Some examples with $H$ not weakly lower semicontinuous}

Instead of $(F3)$ we shall assume now that\\

\begin{itemize}
	\item[(F3')] $I_1$ is weakly lower semicontinuous, bounded on bounded sets, and $I_1(u)\geq C_1\|u\|^{\beta}$  for some $C_1>0$, $\beta>1$, and any $u \in X$.\\
\end{itemize}

\begin{lemma}\label{l3}
Let $c>0$ and assume that $I_2$ is $\alpha$-homogeneous and weakly continuous,  and $I_1$ satisfies $(F1)$ and $(F3')$.
\begin{enumerate}
\item  If $(u_n) \subset \mathcal{N}(\lambda_c)$ is such that $\|u_n\| \to \infty$ or $u_n \rightharpoonup 0$ in $X$, then $\lambda_c(u_n) \to \infty$. In particular, 
$\lambda_c$ is coercive and bounded from below on $\mathcal{N}(\lambda_c)$.
\item  Let $(u_n) \subset \mathcal{S}$. Then $u_n \rightharpoonup 0$ in $X$ if and only if $\lambda_c(t_c(u_n)u_n) \to \infty$. 
In particular $\lambda_c$ satisfies (H3).
\end{enumerate}

\end{lemma}

\begin{proof}\strut
\begin{enumerate}
\item Indeed, repeating the proof of Lemma \ref{l1} we see that if $(u_n) \subset \mathcal{N}(\lambda_c)$ and $\|u_n\| \to \infty$ then  $v_n=\frac{u_n}{\|u_n\|}\rightharpoonup 0$ in $X$. Moreover, by Lemma \ref{l0} and $(F3')$, we have
$$\lambda_c(u_n) \ge \lambda_c(tv_n) \geq \frac{C_1 t^{\beta} -c}{t^{\alpha}I_2(v_n)}, \quad \forall t>0.$$
Chosing $t>0$ such that $C_1 t^{\beta} -c>0$ and letting $n \to \infty$ we obtain $\lambda_c(u_n)\to \infty$. If now $u_n \rightharpoonup 0$ in $X$ then $u_n \not \to 0$ in $X$, and thus $$\lambda_c(u_n) \ge \lambda_c(tu_n) \geq \frac{C_1 t^{\beta}\|u_n\|^{\beta} -c}{t^{\alpha}I_2(u_n)} \ge \frac{C_2 t^{\beta} -c}{t^{\alpha}I_2(u_n)}, \quad \forall t>0,$$
which yields again $\lambda_c(u_n)\to \infty$.\\ 

\item The latter argument in the previous item shows that if $(u_n) \subset \mathcal{S}$ and $u_n \rightharpoonup 0$ in $X$ then $\lambda_c(t_c(u_n)u_n) \to \infty$. Finally, if $u_n \rightharpoonup u \neq 0$ in $X$ then $(t_c(u_n))$ is bounded (by Lemma \ref{l0}) and away from zero, so we can assume that  $t_c(u_n)u_n \rightharpoonup t_0u \neq 0$ in $X$. Since $I_1$ is bounded on bounded sets and $I_2$ is weakly continuous, we infer that 
$$\lambda_c(t_c(u_n)u_n)=\frac{I_1(t_c(u_n)u_n)-c}{I_2(t_c(u_n)u_n)}$$ remains bounded.
\end{enumerate}	
\end{proof}

Let us consider two problems where $c>0$ and  $H$ is not weakly lower semicontinuous, so that Theorem \ref{c1} does not apply.

\subsection*{A generalized $(p,q)$-Laplacian problem}

Our first application deals with the functional
\[
\Phi_\lambda(u)= \frac{1}{p}\displaystyle\int_{\Omega}A(|\nabla u|^p)-\frac{\lambda}{r}\int_\Omega |u|^r, \quad u \in W_0^{1,p}(\Omega),
\]
where $1<r<p^*$. Here $A(t):=\displaystyle\int_{0}^{t}a(s) ds$ for $t \ge 0$, and $a:[0,\infty)\rightarrow
[0,\infty)$ is a $C^{1}$ function satisfying the following conditions:\\
\begin{itemize}
	\item[(A1)] $k_0\left( 1+t^{\frac{q-p}{p}}\right) \leq a(t) \leq k_1\left( 1+t^{\frac{q-p}{p}}\right)$ for every  $t>0$, and some constants $k_0,k_1>0$ and $1<q\le p$ such that $k_1<\frac{r}{p} k_0$.
	\item[(A2)]  $t \mapsto a(t^p)t^p-\frac{r}{p}A(t^p)$ is decreasing in $(0,\infty)$.\\
	\item[(A3)]  $t \mapsto A(t^p)$ is convex in  $(0,\infty)$.\\
\end{itemize}
This choice of $\Phi_\lambda$ corresponds to the quasilinear problem
\begin{equation}\label{quasig}
-div \left(a(|\nabla u|^p)|\nabla u|^{p-2}\nabla u\right) = \lambda |u|^{r-2}u, \quad u \in W_0^{1,p}(\Omega),
\end{equation}
which is a generalized $(p,q)$-Laplacian problem, as it becomes
\begin{equation*}
-\Delta_p u -\Delta_q u = \lambda |u|^{r-2}u, \quad u \in W_0^{1,p}(\Omega),
\end{equation*}
 if we choose $a(t)=1+t^{\frac{q-p}{p}}$. In this case (A1) and (A2) hold with $k_0=k_1=1$, and $r>p$.

\begin{corollary}
Under the previous conditions, let $c>0$.
Then for every $n\in \mathbb{N}$ the problem \eqref{quasig} has a pair of nontrivial solutions $\pm u_{n,c}$ having energy $c$ for $\lambda=\lambda_{n,c}$. Moreover, for $\lambda=\lambda_{1,c}$ the solution $u_{1,c}$ can be chosen nonnegative. Finally, \eqref{quasig} has no solution having energy $c$ for $\lambda<\lambda_{1,c}$.
\end{corollary}

\begin{proof}
We take now $I_1(u)=\frac{1}{p}\int_\Omega A(|\nabla u|^p)$ and $I_2(u)=\frac{1}{r}\int_\Omega |u|^r$ for $u \in X=W_0^{1,p}(\Omega)$, so that $\alpha=r$. Then $H(u)=\int_\Omega \left(a(|\nabla u|^p)|\nabla u|^p -\frac{r}{p}A(|\nabla u|^p)\right)$, and by (A2)
the map $t \mapsto H(tu)$ is decreasing in $(0,\infty)$. Moreover $\displaystyle \lim_{t \to \infty} H(tu)=-\infty$ uniformly on weakly compact subsets of $X \setminus \{0\}$. Indeed, 
note that (A1) yields that $$a(s^p)s^p-\frac{r}{p}A(s^p)\leq k_1(1+s^{q-p})s^p-\frac{r}{p}k_0\left(s^p+\frac{p}{q}s^q\right)=\left(k_1-\frac{r}{p}k_0\right) s^p +\left(k_1-\frac{r}{q}k_0\right) s^q.$$
Thus $$H(tu) \leq \left(k_1-\frac{r}{p}k_0\right) t^p \|u\|^p +\left(k_1-\frac{r}{q}k_0\right)t^q \int_\Omega |\nabla u|^q.$$
Since $k_1-\frac{r}{p}k_0<0$ and $q\le p$, we obtain the desired conclusion. Lemma \ref{l0} yields that $\lambda_c$ satisfies (H1).
Note also that $I_1$ is weakly lower semicontinuous by $(A3)$, and  by $(A1)$ it is bounded on bounded sets, and satisfies
$I_1(u)\ge k_0 \|u\|^p$ for any $u \in X$. Thus (F3') is satisfied, and by Lemma \ref{l3} we infer that $\lambda_c$ is bounded from below on $\mathcal{N}(\lambda_c)$ and satisfies (H3).

Finally, if $(u_n) \subset \mathcal{N}(\lambda_c)$ is a Palais-Smale sequence for $\lambda_c$, then arguing as in Lemma \ref{l1}(4) we find that $(u_n)$ is bounded and $\Phi_{\lambda_c(u_n)}'(u_n) \to 0$. By Lemma \ref{l3} we know that $u_n \not \rightharpoonup 0$. The complete continuity of $I_2$ yields that $I_1'(u_n) \to 0$. The strict convexity of $t \mapsto A(t^p)$ implies that $I_1'$ belongs to the class $(S_+)$ cf. \cite[Lemma 2.7]{DU}, and consequently $(u_n)$ has a convergent subsequence in $X \setminus \{0\}$. Therefore $\lambda_c$ satisfies the Palais-Smale condition on $\mathcal{N}(\lambda_c)$, and Theorem \ref{tn} yields the conclusion.
\end{proof}

\subsection*{The Brezis-Nirenberg problem}

We consider now the problem \eqref{semi} with $f(s)=|s|^{2^*-1}s$, so that $F(s)=\frac{1}{2^*}|s|^{2^*}$ for $s \in \mathbb{R}$. Let $\lambda_c$ be given by \eqref{lex}, i.e.
 \[
\lambda_c(u)=\frac{\int_\Omega |\nabla u|^2-\frac{2}{2^*}\int_\Omega |u|^{2^*} -2c}{\int_\Omega |u|^2},
\]
for $u \in H_0^1(\Omega) \setminus \{0\}$.
Then $$H(u)=\int_\Omega \left(2F(u)-f(u)u\right) = \frac{2-2^*}{2^*} \int_\Omega |u|^{2^*}=-\frac{2}{N}\|u\|_{2^*}^{2^*},$$ so that
\begin{equation}\label{ncc}
\mathcal{N}(\lambda_c)=\{u \in H_0^1(\Omega): \|u\|_{2^*}^{2^*}=Nc\}.
\end{equation}
Thus for any $u \in X \setminus \{0\}$ the map $t \mapsto H(tu)$ is decreasing in $(0,\infty)$. Moreover, since $H$ is weakly upper semicontinuous, we see that $\displaystyle \lim_{t \to \infty} H(tu)=-\infty$ uniformly on weakly compact subsets of $X \setminus \{0\}$. Therefore  $(F1)$ holds, and Lemma \ref{l0} implies that $\lambda_c$ satisfies (H1).

\begin{lemma}
	Let $c>0$. Then:
\begin{enumerate}
\item $\lambda_c$ is coercive on  $\mathcal{N}(\lambda_c)$.
\item Assume in addition $c<\frac{S^{\frac{N}{2}}}{N}$, where $S$ is the best constant of the embedding $H_0^1(\Omega) \subset L^{2^*}(\Omega)$. If $(u_n) \subset \mathcal{N}(\lambda_c)$ with $u_n \rightharpoonup 0$ in $X$, then  $\lambda_c(u_n) \to \infty$. In particular $\lambda_c$ is bounded from below on  $\mathcal{N}(\lambda_c)$. Moreover, for any $(u_n) \subset \mathcal{S}$ we have $u_n \rightharpoonup 0$ in $X$ if and only if $\lambda_c(t_c(u_n)u_n) \to \infty$.
\end{enumerate}	
\end{lemma}

\begin{proof}\strut
\begin{enumerate}
\item From \eqref{ncc} we have
$$\lambda_c(u)=\frac{\|u\|^2-\frac{2}{2^*} \|u\|_{2^*}^{2^*} -2c}{\|u\|_2^2}=\frac{\|u\|^2-\frac{2}{2^*}Nc -2c}{\|u\|_2^2}, \quad \forall u \in \mathcal{N}(\lambda_c).$$
Since $\mathcal{N}(\lambda_c)$ is bounded in $L^2(\Omega)$ we deduce that $\lambda_c$ is coercive therein.\\

 \item  Let $(u_n) \subset \mathcal{N}(\lambda_c)$ with $u_n \rightharpoonup 0$ in $X$. Then
 \begin{eqnarray*}
 \lambda_c(u_n)\ge \lambda_c(tu_n)&=&\frac{\|t u_n\|^2-\frac{2}{2^*} \|t u_n\|_{2^*}^{2^*} -2c}{\|t u_n\|_2^2}=\frac{\|u_n\|^2-\frac{2}{2^*}Nct^{2^*-2} -2ct^{-2}}{\|u_n\|_2^2}\\
 &\ge & \frac{S^{-1}(Nc)^{\frac{2}{2^*}}-\frac{2}{2^*}Nct^{2^*-2} -2ct^{-2}}{\|u_n\|_2^2}=t^{-2} \frac{j(t)}{\|u_n\|_2^2},
 \end{eqnarray*}
where $j(t):=S^{-1}(Nc)^{\frac{2}{2^*}}t^2-\frac{2}{2^*}Nct^{2^*} -2c$. Since $u_n \to 0$ in $L^2(\Omega)$ it suffices to have $\displaystyle \max_{t>0} j(t)>0$ to obtain the desired conclusion. A simple computation shows that $$\max_{t>0} j(t)=\frac{2^*-2}{2^*}S^{\frac{2^*}{2^*-2}}-2c,$$ so that $\displaystyle \max_{t>0} j(t)>0$ if and only if $c<\frac{S^{\frac{N}{2}}}{N}$.

Let now $(u_n) \subset \mathcal{S}$. The latter argument shows that $\lambda_c(t_c(u_n)u_n) \to \infty$ if $u_n \rightharpoonup 0$ in $X$. Finally, if $u_n \rightharpoonup u \neq 0$ then Lemma \ref{l0} yields that $(t_c(u_n))$ is bounded. Moreover, we know that $(t_c(u_n))$ is away from zero, so
$$\lambda_c(t_c(u_n)u_n)=\frac{1-\frac{2}{2^*}Nct_c(u_n)^{2^*-2} -2ct_c(u_n)^{-2}}{\|u_n\|_2^2} \leq C,$$ 
for some $C>0$.
\end{enumerate}	
\end{proof}

\begin{lemma}
Let $c<\frac{S^{\frac{N}{2}}}{N}$. Then $\lambda_c$ satisfies the Palais-Smale condition on $ \mathcal{N}(\lambda_c)$ at any level $d \in \mathbb{R}$. As a consequence, the conclusions of Corollary \ref{c2} remain valid.
\end{lemma}

\begin{proof}
Let $(u_n) \subset \mathcal{N}(\lambda_c)$ be a Palais-Smale sequence for $\lambda_c$. Since $(u_n)$ is bounded in $L^{2^*}(\Omega)$ we see that $I_2(u_n)$ and $I_2'(u_n)$ are bounded if $I_2(u)=\frac{1}{2}\|u\|_2^2$. Hence by Lemma \ref{l00} we deduce that $(u_n)$ is a Palais-Smale sequence for $\Phi_d$ at the level $c$. Since $\Phi_d$ satisfies the Palais-Smale condition for $c<\frac{S^{\frac{N}{2}}}{N}$ (cf. \cite{BN}) we infer that $(u_n)$ has a convergent subsequence in $X \setminus \{0\}$. Thus we can apply Theorem \ref{tn} to get the desired conclusion.

\end{proof}

\medskip
\section{A problem involving the affine $p$-Laplacian}
\medskip

Let us now deal with the problem \eqref{P}. We shall apply Theorem \ref{tn} to the functional
\[
\Phi_{\cal A}(u) = \frac{1}{p} {\cal E}^p_{p,\Omega}(u) -  \int_\Omega F(u),  \quad u\in W^{1,p}_0(\Omega) \setminus \{0\},
\]
where ${\cal E}_{p,\Omega}$ is given by \eqref{dep}.
We set 
$$\mathcal{N}_{\cal A}:=\{u \in W^{1,p}_0(\Omega)  \setminus \{0\}:\, \Phi_{\cal A}'(u)u=0\}$$
and
\[
\mathcal{S}^{{\cal A}} := \{u \in W^{1,p}_0(\Omega):\, {\cal E}_{p,\Omega}(u) = 1\}.
\]

\begin{theorem}\label{tap}
Let $f \in C(\mathbb{R})$ satisfy the following conditions:
\begin{enumerate}
\item $|f(s)| \leq C(1+|s|^{r-1})$ for some $C>0$ and $1<r<p^*$, and every $s \in \mathbb{R}$. 
\item $f(s)=o(|s|^{p-1})$ as $s \to 0$.
\item $\frac{f(s)}{|s|^{p-1}}$ is strictly increasing in $\mathbb{R} \setminus \{0\}$.
\item $\frac{F(s)}{|s|^p} \to \infty$ as $|s| \to \infty$.
\end{enumerate}
Then the functional $\Phi_{\cal A}$ has a ground state level, which is achieved by a ground state solution $u_0$ of \eqref{P}. If, in addition, $f$ is odd, then we can choose $u_0 \geq0$ and $\Phi_{\cal A}$ has infinitely many pairs of critical points, which are solutions of \eqref{P}.
\end{theorem}

The proof of Theorem \ref{tap} follows from the next lemmas:

\begin{lemma}\label{la}
Under the assumptions of Theorem \ref{tap} the following assertions hold:
\begin{enumerate}
\item For any $u \in W^{1,p}_0(\Omega) \setminus \{0\}$ the map $t \mapsto \Phi_{\cal A}(tu)$ has exactly one critical point $t(u)$, which is a global maximum point. Moreover, the map $u\mapsto t(u)$ is $(-1)$-homogeneous, bounded from below by a positive constant in $\mathcal{S}^{{\cal A}}$, and bounded from above in any compact set $W \subset \mathcal{S}^{{\cal A}}$. 
\item The {\it affine} $p$-energy ${\cal E}^p_{p,\Omega}$ is bounded from below by a positive constant in $\mathcal{N}_{\cal A}$. In particular $\mathcal{N}_{\cal A}$ is away from zero in $L^q(\Omega)$ for any $1<q<p^*$.
\end{enumerate}
\end{lemma}

\begin{proof}\strut
\begin{enumerate}
\item The first assertion follows from the $p$-homogeinety of ${\cal E}^p_{p,\Omega}(u)$ and the conditions on $f$. Let us show that $t(su)=s^{-1}t(u)$ for any $u \neq 0$ and $s>0$. Indeed, writing $\phi_u(t)=\Phi_{\cal A}(tu)$ we see that $\phi_{su}'(t)=s\phi_u'(st)$, which yields that $st(su)=t(u)$, as desired. Finally, if $u \in \mathcal{S}^{{\cal A}}$ then
$t(u)u \in \mathcal{N}_{\cal A}$ provides us with $$t(u)^{p-1}=\int_\Omega f(t(u)u)u,$$ which combined with the behavior of $f$ at zero and infinity yields the final assertions.\\

\item By the previous item we know that
$$t(u)=\frac{1}{{\cal E}_{p,\Omega}(u)}t\left(\frac{u}{{\cal E}_{p,\Omega}(u)}\right).$$
Thus, for $u \in \mathcal{N}_{\cal A}$ there holds
$1=t(u)=\frac{1}{{\cal E}_{p,\Omega}(u)}t\left(\frac{u}{{\cal E}_{p,\Omega}(u)}\right)$, i.e. ${\cal E}_{p,\Omega}(u)=t\left(\frac{u}{{\cal E}_{p,\Omega}(u)}\right)$ and since $\frac{u}{{\cal E}_{p,\Omega}(u)} \in \mathcal{S}^{{\cal A}}$ the previous item yields the first assertion. The second assertion follows from the previous one and the equality
$${\cal E}^p_{p,\Omega}(u)=\int_\Omega f(u)u$$ for $u \in \mathcal{N}_{\cal A}$.
\end{enumerate}
\end{proof}

\begin{lemma}
Under the assumptions of Theorem \ref{tap} the functional $\Phi_{\cal A}$ satisfies (H1) and (H2).
\end{lemma}

\begin{proof}
Let us first deal with (H1). From the $p$-homogeinety of ${\cal E}^p_{p,\Omega}(u)$ and the conditions on $f$ it is clear that for any $u \in W^{1,p}_0(\Omega) \setminus \{0\}$ the map $t \mapsto \Phi_{\cal A}(tu)$ has exactly one critical point $t(u)$, which is a global maximum point. It remains to show that the map $u\mapsto t(u)$ is bounded from below by a positive constant in $\mathcal{S}$, and bounded from above in any compact set $W_1 \subset \mathcal{S}$. Indeed, by Lemma \ref{la} there exists $\delta > 0$ such that $t(u)\geq \delta$ for all $u\in \mathcal{S}^{{\cal A}}$. Since ${\cal E}_{p,\Omega}(u)\leq 1$ for any  $u \in \mathcal{S}$ we find that 
\begin{equation} \label{etu}
t(u)=\frac{1}{{\cal E}_{p,\Omega}(u)}t\left(\frac{u}{{\cal E}_{p,\Omega}(u)}\right)\geq t\left(\frac{u}{{\cal E}_{p,\Omega}(u)}\right) \geq\delta
\end{equation} for any $u\in \mathcal{S}$. Now, let $W_1 \subset \mathcal{S}$ be a compact set. We set $W_2:=\left\{\frac{u}{{\cal E}_{p,\Omega}(u)}:u\in W_1\right\}$, so that $W_2\subset \mathcal{S}^{{\cal A}}$ is a compact set as well, and by Lemma \ref{la} we know that $u \mapsto t(u)$ is bounded from above in $W_2$. Since $W_1 \subset \mathcal{S}$ is a compact set, there exists $\delta_1>0$ such that ${\cal E}_{p,\Omega}(u)\geq\delta_1$ in $W_{1}$. Thus \eqref{etu} shows that there exists $\delta_2>0$ such that $t(u)\leq\delta_2$ for all $u\in W_{1}$. 

Next we prove (H2). It is clear that $\Phi_{\cal A}$ is even and positive on $\mathcal{N}_{\cal A}$.
Let $(u_k)$ be a sequence in $\mathcal{N}_{\cal A}$ such that $(\Phi_{\cal A}(u_k))$ is bounded and $\Vert \Phi'_{\cal A}(u_k) \Vert_{W^{-1, p'}(\Omega)} \rightarrow 0$. 
We claim that $({\cal E}_{p,\Omega}(u_k))$ is bounded. Otherwise ${\cal E}_{p,\Omega}(u_k)\rightarrow\infty$, and we set $v_k := \frac{u_k}{{\cal E}_{p,\Omega}(u_k)}$. Thus, up to a subsequence, we have $v_k \rightarrow v$ strongly in $L^q(\Omega)$, for $1\leq q<p^*$. Suppose $v = 0$. Since $u_k\in\mathcal{N}_{\cal A}$  we have $u_k =t(v_k)v_k$. Then for each $t > 0$, we obtain
\[
d \geq \Phi_{\cal A}(u_k) = \Phi_{\cal A}(t(v_k)v_k)\geq \Phi_{\cal A}(tv_k)\geq \frac{t^p}{p} - \int_\Omega F(tv_k) \rightarrow \frac{t^p}{p}.
\]
This yields a contradiction upon choosing $t > (dp)^{1/p}$. So $v\neq 0$ and hence
\[
0 \leq\frac{\Phi_{\cal A}(u_k)}{{\cal E}_{p,\Omega}^p(u_k)}\leq\frac{1}{p}-\int_\Omega \frac{ F({\cal E}_{p,\Omega}(u_k)v_k)}{{\cal E}_{p,\Omega}^p(u_k)}\rightarrow -\infty
\]
as $k\rightarrow\infty$ by assumption (4), and we reach a contradiction again. It follows that $({\cal E}_{p,\Omega}(u_k))$ is
bounded, so that by the affine Rellich-Kondrachov compactness theorem (see Theorem 6.5.3 in \cite{T}) we can assume that $u_k \rightarrow u$ in $L^q(\Omega)$, for $1\leq q<p^*$.
By Lemma \ref{la}(2) we know that $u \neq 0$. Hence, by Corollary 2.1 of \cite{LM1}, $(u_k)$ is bounded in $W_0^{1,p}(\Omega)$, so that $u_k \rightharpoonup u \neq 0$ in $W_0^{1,p}(\Omega)$, up to a subsequence. Moreover, $\Vert \Phi'_{\cal A}(u_k) \Vert_{W^{-1, p'}(\Omega)} \rightarrow 0$ clearly implies

\[
\lim_{k \rightarrow \infty}\, \Phi'_{\cal A}(u_k)(u_k - u) = 0,
\]
which by standard arguments yields
\[
\lim_{k \rightarrow \infty}\, \langle \Delta^{\cal A}_p u_k, u_k - u \rangle = 0.
\]
Finally, by Theorem 2 of \cite{LM2}, $u_k \to u$ in $W_0^{1,p}(\Omega)$, and the proof is complete.
\end{proof}

\begin{remark}\label{raf}
Arguing as in Remark 17 of \cite{SW} we can show that $u_0 \geq 0$ or $u_0 \leq 0$. Thus if $f$ is odd then we can choose $u_0 \geq 0$. If, in addition, $\Omega$ has a smooth boundary then $u_0>0$ in $\Omega$, by Prop. 3.2 in \cite{LM1}. 
\end{remark}

\medskip
\section{A degenerate Kirchhoff type problem}
\medskip

Finally we deal with the functional
\[
\Phi(u) = \frac{1}{\theta} \|u\|^{\theta}- \int_\Omega F(u), \quad u\in W^{1,p}_0(\Omega) \setminus \{0\},
\]
which is associated to the problem \eqref{P0}.
We consider two cases with respect to $\theta$, namely, $0<\theta\leq 1$ and $\theta<0$.
\begin{theorem}\label{tap1}\strut
Let $0<\theta\leq 1$ and $f \in C(\mathbb{R})$ satisfy:
\begin{enumerate}
\item $|f(s)| \leq C(1+|s|^{r-1})$ for some $C>0$ and $1<r<p^*$, and every $s \in \mathbb{R}$. 
\item $f(s)=o(|s|^{\theta-1})$ as $s \to 0$.
\item $\frac{f(s)}{|s|^{\theta-1}}$ is strictly increasing in $\mathbb{R} \setminus \{0\}$.
\item $\frac{F(s)}{|s|^\theta} \to \infty$ as $|s| \to \infty$.

\end{enumerate}
Then the functional $\Phi$ has a ground state level,  which is achieved by a ground state solution $u_0$ of \eqref{P}. If, in addition, $f$ is odd, then we can choose $u_0 >0$ and $\Phi$ has infinitely many pairs of critical points, which are solutions of \eqref{P0}. Moreover, $\lambda_n\rightarrow\infty$.
\end{theorem}

The proof of Theorem \ref{tap1} follows from the next lemmas:

\begin{lemma}\label{lf}
Under the assumptions of Theorem \ref{tap1} the following assertions hold:
\begin{enumerate}
\item $\Phi$ satisfies (H1). 
\item The function $\Vert\cdot\Vert$ is bounded from below by a positive constant in $\mathcal{N}$. In particular $\mathcal{N}$ is away from zero in $L^q(\Omega)$ for any $1<q<p^*$.
\end{enumerate}
\end{lemma}

\begin{proof}\strut
\begin{enumerate}
\item It follows in a similar way as Lemma \ref{la}(1).\\

\item Note that
$$t(u)=\frac{1}{\Vert u\Vert}t\left(\frac{u}{\Vert u\Vert}\right).$$
Then, $\Vert u\Vert=t\left(\frac{u}{\Vert u\Vert}\right)$ for all $u \in \mathcal{N}$. Since $\frac{u}{\Vert u\Vert} \in \mathcal{S}$, we obtain the first assertion by (1). For the second assertion note that for $u \in \mathcal{N}$ we have
$$\Vert u\Vert^\theta=\int_\Omega f(u)u.$$
\end{enumerate}
\end{proof}

\begin{lemma}\label{lff}
$\Phi$ satisfies the Palais-Smale condition on $\mathcal{N}$.
\end{lemma}
\begin{proof}
Let $(u_k)$ be a sequence in $\mathcal{N}$ such that $(\Phi(u_k))$ is bounded and $\Vert \Phi'(u_k) \Vert_{W^{-1, p'}(\Omega)} \rightarrow 0$. We claim that $(u_k)$ is bounded in $W_0^{1,p}(\Omega)$. Suppose by contradiction that $\Vert u_k\Vert\rightarrow\infty$, and we take $v_k := \frac{u_k}{\Vert u_k\Vert}$. Thus, $v_k \rightarrow v$ strongly in $L^q(\Omega)$, for $1\leq q<p^*$ module a subsequence. Assume $v = 0$. Since $u_k\in\mathcal{N}$  we have $u_k =t(v_k)v_k$. Then
\begin{equation}\label{13}
d \geq \Phi(u_k) = \Phi(t(v_k)v_k)\geq \Phi(tv_k)= \frac{t^\theta}{\theta} - \int_\Omega F(t v_k) \rightarrow \frac{t^\theta}{\theta}
\end{equation}
for all $t > 0$. Choosing $t > (d\theta)^{1/\theta}$, we obtain a contradiction. So $v\neq 0$ and hence
\[
0 \leq\frac{\Phi(u_k)}{\Vert u_k\Vert^{\theta}}=\frac{1}{\theta}-\int_\Omega \frac{F(\Vert u_k\Vert v_k)}{\Vert u_k\Vert^{\theta}}\rightarrow -\infty
\]   
as $k\rightarrow\infty$ by assumption (4), and we derive a contradiction again. It follows that $(u_k)$ is
bounded, so that $u_k \rightharpoonup u$ in $W_0^{1,p}(\Omega)$, module a subsequence. If $u = 0$, we see as in (\ref{13}) that
\[
d \geq \Phi(u_k)\geq \Phi(tu_k)\geq  \frac{d_0t^\theta}{\theta} - \int_\Omega F(t u_k) \rightarrow \frac{d_0t^\theta}{\theta}
 \]
for all $t > 0$, where $d_0 = \inf_{\mathcal{N}} \Vert u\Vert^\theta > 0$, a contradiction. Hence $u\neq 0$. Since $\Vert \Phi'(u_k) \Vert_{W^{-1, p'}(\Omega)} \rightarrow 0$ implies

\[
\lim_{k \rightarrow \infty}\, \Phi'(u_k)(u_k - u) = 0,
\]
we get
\[
\lim_{k \rightarrow \infty}\, \Vert u_k\Vert^{\theta-p}\langle \Delta_p u_k, u_k - u \rangle = 0,
\]
so that
\[
\lim_{k \rightarrow \infty}\, \langle \Delta_p u_k, u_k - u \rangle = 0.
\]
By the $(S_+)$ property of the $p$-Laplacian we find that $u_k \to u$ in $W_0^{1,p}(\Omega)$, which finishes the proof.
\end{proof}

\begin{lemma}
$\Phi$ satisfies (H3).
\end{lemma}

\begin{proof}
Let $(u_k) \subset \mathcal{S}$ be a sequence such that $u_k\rightharpoonup 0$ in $W_0^{1,p}(\Omega)$. Then, we have $u_k \rightarrow 0$ strongly in $L^q(\Omega)$, for $1\leq q<p^*$. Thus
\[
 \Phi(t(u_k)u_k)\geq \Phi(tu_k)=  \frac{t^\theta}{\theta} - \int_\Omega F(t u_k) \rightarrow \frac{t^\theta}{\theta},
 \]
 for all $t>0$. Therefore, $\Phi(t(u_k)u_k)\rightarrow\infty$ as $k\rightarrow\infty$.
\end{proof}

Let us deal now with the case $\theta<0$:
\begin{theorem}\label{tap2}
	Let $\theta<0$ and $f \in C(\mathbb{R})$ satisfy the assumptions of Theorem \ref{tap1}. 
Then the functional $\Phi$ has a ground state level,  which is achieved by a ground state solution $u_0$ of \eqref{P}. If, in addition, $f$ is odd, then we can choose $u_0 >0$ and $\Phi$ has infinitely many pairs of critical points, which are solutions of \eqref{P0}.
\end{theorem}

The next lemma can be proved exactly as Lemma \ref{lf}:
\begin{lemma}
	Under the assumptions of Theorem \ref{tap2} the following assertions hold:
	\begin{enumerate}
		\item $\Phi$ satisfies (H1). 
		\item The function $\Vert\cdot\Vert$ is bounded from below by a positive constant in $\mathcal{N}$. 
	\end{enumerate}
\end{lemma}

\begin{lemma}
	$\Phi$ satisfies (H2).
\end{lemma}
\begin{proof} First we prove that $\Phi$ is bounded from below on $\mathcal{N}$. Indeed, assume by contradiction that $\Phi(u_k) \to -\infty$ with $(u_k) \subset \mathcal{N}$.
Since $\mathcal{N}$ is away from zero, we must have $\|u_k\| \to \infty$. Set $v_k:=\frac{u_k}{\|u_k\|}$ and assume that $v_k \rightharpoonup v$. Since
$$\Phi(u_k) \geq \Phi(tv_k)= \frac{t^\theta}{\theta} - \int_\Omega F(t v_k)  \quad \forall t>0,$$
we see that $v \neq 0$, otherwise $\Phi(u_k) \geq \frac{t^\theta}{\theta}$ for any $t>0$, which contradicts $\Phi(u_k) \to -\infty$.
On the other hand, from $$\|u_k\|^{\theta}=\int_\Omega f(u_k)u_k$$ we find that $$1=\int_\Omega \frac{f(\|u_k\|v_k)}{\|u_k\|^{\theta-1}|v_k|^{\theta-1}}|v_k|^{\theta-1}v_k \to \infty,$$
which yields another contradiction. Thus $\Phi$ is bounded from below on $\mathcal{N}$.


	
	Let us show that $\Phi$ satisfies the Palais-Smale condition on $\mathcal{N}$. First note that $\Phi$ takes negative values, so that $\lambda_k<0$ for every $k$. Hence we shall prove that the Palais-Smale condition holds at negative levels.	
	 Let $(u_k)$ be a sequence in $\mathcal{N}$ such that $\Phi(u_k) \to d<0$ and $\Vert \Phi'(u_k) \Vert_{W^{-1, p'}(\Omega)} \rightarrow 0$. Arguing as above it is clear that $(u_k)$ is bounded in $W_0^{1,p}(\Omega)$. The proof can now be concluded as in Lemma \ref{lff}.
\end{proof}

   \begin{figure}[h]
	\centering
	\begin{tikzpicture}[scale=0.8]
		\draw[->] (0,0) -- (4,0) node[below] {\scalebox{0.8}{$s$}};
		
		\draw[->] (0,0) -- (0,3) node[right] {\scalebox{0.8}{$f(s)$}};
		
		\draw[blue,thick,domain=0:1,smooth,variable=\x] plot ({\x},{(\x)^(0.5)});	
		
		\draw[blue,thick,domain=1:4,smooth,variable=\x] plot ({\x},{(\x)^(-1)});	
		
	\end{tikzpicture}
	\caption{Plot of some $f$ satisfying the conditions of Theorem \ref{tap2}.} 
\end{figure}
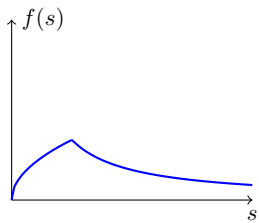

Note that Remark \ref{raf} also applies here, with no need of any smoothness condition on the boundary of $\Omega$.

\begin{example}
Some examples of $f$ satisfying the assumptions of Theorems \ref{tap1} and \ref{tap2} are given by the odd extensions of the following functions:
\begin{enumerate}
	\item $f(s)=s^{\beta}$ for $0 \leq s \leq 1$ and $f(s)=s^{r}$ for $s\geq 1$, where $\beta >0$ and  $\theta-1<r<p^*-1$.
	\item $f(s)=s^{\beta}$ for $0 \leq s \leq 1$ and $f(s)=s^{\theta-1} \frac{\ln(s+1)}{\ln 2}$ for $s\geq 1$, where $\beta >0$.
\end{enumerate}
Note that when $\theta <1$ and $r<0$ these functions vanish at infinity.  
\end{example}

\medskip

{\bf Acknowledgments}
E.J.F. Leite has been supported by CNPq-Grant 316526/2021-5. H. Ramos Quoirin has been supported by FAPEG Programa Pesquisador Visitante Estrangeiro 2024. K. Silva has been supported by  CNPq-Grant 308501/2021-7.  This work was completed during a visit of the second author at the Instituto de Matem\'atica e Estat\'istica, Universidade Federal de Goi\'as, whose warm hospitality is greatly appreciated. \\

{\bf Data availability statement}
Data sharing not applicable to this article as no datasets were generated or analyzed during the current study.

\end{document}